\documentclass[10pt,openany]{article}

\usepackage[utf8]{inputenc}
\usepackage[a4paper,hmargin={4cm,4cm},vmargin={3.5cm,3.5cm}]{geometry}
\usepackage[colorlinks=true,bookmarksopen=false,linkcolor=blue,citecolor=red]{hyperref}

\usepackage{marvosym}
\usepackage[shadow, textsize=footnotesize,color=blue!40
, bordercolor=blue!80]{todonotes}

\usepackage{subfiles}
\usepackage{ulem}
\usepackage{calc} 
\usepackage{comment}
\usepackage{amssymb}
\usepackage{amsthm}
\usepackage{amsmath}
\usepackage{amsrefs}
\usepackage{enumitem}
\usepackage{titlesec}
\usepackage{titletoc}
\usepackage{hyperref}
\usepackage{dsfont}
\usepackage{euscript}
\usepackage{fourier-orns}
\usepackage{palatino}
\usepackage[scr]{rsfso}
\usepackage[sc]{mathpazo} 
\usepackage[T1]{fontenc}
\usepackage{graphicx}
\usepackage{tikz-cd}
\tikzcdset{
	arrow style=tikz,
	diagrams={>=latex}
}
\definecolor{blue}{rgb}{0, 0.1, 0.6}

\DeclareFontFamily{OT1}{pzc}{}
\DeclareFontShape{OT1}{pzc}{m}{it}{<-> s * [1.10] pzcmi7t}{}
\DeclareMathAlphabet{\mathpzc}{OT1}{pzc}{m}{it}

\def\emph#1{\textcolor{blue}{\textbf{\boldmath #1}}}

\providecommand{\keywords}[1]
{
  \textbf{\textit{Keywords---}} #1
}
\providecommand{\msc}[1]
{
 \textbf{\textit{MSC---}} #1
}

\newcommand{\cev}[1]{\reflectbox{\ensuremath{\vec{\reflectbox{\ensuremath{#1}}}}}}

\def\models{\vDash}
\def\pmodels{\mathrel{\models\kern-1.5ex\raisebox{.5ex}{*}}}

\def\FIN{{\rm FIN}}

\def\ran{\mathop{\rm ran}}

\def\Aut{\textrm{Aut\kern.15ex}}
\def\Autf{\mathord{\rm Aut\kern.15ex{f}\kern.15ex}}
\def\Cb{\textrm{Cb\kern.15ex}}

\definecolor{blue}{rgb}{0, 0.1, 0.6}

\def\bl{\color{black}}
\def\nonfork{\mathop{\raise0.2ex\hbox{\ooalign{\hidewidth$\vert$\hidewidth\cr\raise-0.9ex\hbox{$\smile$}}}}}

\def\cnonfork{\mathbin{\raise1.8ex\rlap{\kern0.6ex\rule{0.6ex}{0.1ex}}\rlap{\kern1.1ex\rule{0.1ex}{1.9ex}}\raise-0.3ex\hbox{$\smile$} } }

\def\cpaw{\mathbin{\ooalign{\kern-0.4ex$-$\hidewidth\cr$<$}}}
\def\cpawdot{\ooalign{$\kern1.2ex\cdot$\cr$\cpaw$\cr}}

\def\nonforkc{\mathbin{\raise1.8ex\rlap{\kern1.1ex\rule{0.6ex}{0.1ex}}\rlap{\kern1.1ex\rule{0.1ex}{1.9ex}}\raise-0.3ex\hbox{$\smile$} } }

\def\rela{\mathbin{\ooalign{\kern-0.4ex$-$\hidewidth\cr$<$}}}

\def\Aut{\textrm{Aut\kern.15ex}}
\def\Autf{\mathord{\rm Aut\kern.15ex{f}\kern.15ex}}
\def\Cb{\textrm{Cb\kern.15ex}}

    \DeclareMathOperator{\powerset}{\mathscr{P}}

\def\nonfork{\mathop{\raise0.2ex\hbox{
			\ooalign{\hidewidth$\vert$\hidewidth\cr\raise-0.9ex\hbox{$\smile$}}}}}

\def\cnonfork{\mathbin{\raise1.8ex\rlap{\kern0.6ex\rule{0.6ex}{0.1ex}}
		\rlap{\kern1.1ex\rule{0.1ex}{1.9ex}}\raise-0.3ex\hbox{$\smile$} } }

\def\cpaw{\mathbin{\ooalign{\kern-0.4ex$-$\hidewidth\cr$<$}}}
\def\cpawdot{\ooalign{$\kern1.2ex\cdot$\cr$\cpaw$\cr}}

\def\isomap{\rlap{\kern0.8ex\raisebox{1ex}{\scriptsize$\sim$}}\rightarrow}

\def\R{\EuScript R}

\def\G{\EuScript G}

\def\<{\langle}
\def\>{\rangle}
\def\0{\varnothing}
\def\theta{\vartheta}
\def\phi{\varphi}
\def\epsilon{\varepsilon}

\def\isomap{\rlap{\kern0.8ex\raisebox{1ex}{\scriptsize$\sim$}}\rightarrow}

\def\R{\mathcal R}

\def\G{\EuScript G}

\def\<{\langle}
\def\>{\rangle}
\def\0{\varnothing}
\def\theta{\vartheta}
\def\phi{\varphi}
\def\epsilon{\varepsilon}

\titlecontents{section}
[3.8em] 
{\vskip-1ex}
{\contentslabel{1.5em}}
{\hspace*{-2.3em}}
{\titlerule*[1pc]{}\contentspage}

\titleformat{\section}[block]{\Large\bfseries}{\makebox[5ex][r]{\textbf{\thesection}}}{1.5ex}{}
\titlespacing*{\chapter}{0em}{.5ex plus .2ex minus .2ex}{2.3ex plus .2ex}
\titlespacing*{\section}{-9.7ex}{3ex plus .5ex minus .5ex}{1ex plus .2ex minus .2ex}

\newtheoremstyle{mio}
{2\parskip}
{\parskip}
{\sl}
{}
{\bfseries}
{}
{1ex}
{\llap{\thmnumber{#2}\hskip2mm}\thmname{#1}\thmnote{\bfseries{} #3}}

\newtheoremstyle{liscio}
{2\parskip}
{0mm}
{}
{}
{\bfseries}
{}
{1.5ex}
{\llap{\thmnumber{#2}\hskip2mm}\thmname{#1}\thmnote{\bfseries{} #3}}

\newcounter{thm}[section]

\theoremstyle{mio}
\newtheorem{theorem}[thm]{Theorem}
\newtheorem{Main 1}[thm]{Theorem (main theorem 1)}
\newtheorem{Main 2}[thm]{Theorem (main theorem 2)}
\newtheorem*{theorem*}{Theorem}

\newtheorem{corollary}[thm]{Corollary}
\newtheorem{proposition}[thm]{Proposition}\newtheorem*{proposition*}{Proposition}
\newtheorem{lemma}[thm]{Lemma}
\newtheorem{fact}[thm]{Fact}
\newtheorem{OP}[thm]{Open Problem}
\newtheorem{definition}[thm]{Definition}

\theoremstyle{liscio}

\newtheorem{example}[thm]{Example}
\def\QED{\noindent\nolinebreak[4]\hspace{\stretch{1}}\rlap{\ \ $\Box$}\medskip}
\renewenvironment{proof}[1][Proof]%
{\begin{trivlist}\item[\hskip\labelsep {\bf #1}]}
	{\QED\end{trivlist}}

\definecolor{violet}{RGB}{115, 0, 205}
\definecolor{brown}{RGB}{150, 50, 10}
\definecolor{green}{RGB}{5,110, 35}
\definecolor{emphcolor}{rgb}{.98,.98,.70}

\renewcommand*{\emph}[1]{ \smash{\tikz[baseline]\node[rectangle, fill=emphcolor, rounded corners, inner xsep=.5ex, inner ysep=.2ex, anchor=base, minimum height = 3ex]{#1};}}

\renewcommand{\L}{\mathcal{L}}
\renewcommand{\H}{\mathcal{H}}
\newcommand{\J}{\mathcal{J}}
\newcommand{\D}{\mathcal{D}}
\newcommand{\K}{\mathcal{K}}

\newcommand{\NN}{\mathbb{N}}

\newcommand{\XX}{\mathbb X}
\newcommand{\YY}{\mathbb{Y}}

\renewcommand{\subset}{\subseteq}

\DeclareMathOperator{\Rbin}{\mathbin{\R}}
\DeclareMathOperator{\Lbin}{\mathbin{\L}}
\DeclareMathOperator{\Hbin}{\mathbin{\H}}

\DeclareMathOperator{\Jbin}{\mathbin{\J}}
\DeclareMathOperator{\conc}{\mathbin{{}^\smallfrown}}
\DeclareMathOperator{\tp}{tp}
\DeclareMathOperator{\fp}{fp}
\DeclareMathOperator{\cw}{cw}

\begin{document}

	\raggedbottom
	\begin{center}
		{\huge\bfseries Ramsey monoids\\[3ex] \normalfont\normalsize 
			Claudio Agostini$^\dag$ and Eugenio Colla$^\dag$\vskip-1ex 
			
			\today}
	\end{center}

	\def\medrel#1{\parbox[t]{6ex}{$\displaystyle\hfil #1$}}
	
	\bigskip\hfil
	\parbox{0.9\textwidth}{
		\textbf{Abstract} \ 
		 Recently, Solecki introduced the notion of Ramsey monoid to produce a common generalization to theorems such as Hindman's theorem, Carlson's theorem, and Gowers' $\FIN_k$ theorem. 
		He proved that an entire class of finite monoids is Ramsey. 
		Here we improve this result, enlarging this class and finding a simple algebraic characterization of finite Ramsey monoids.  
		We extend in a similar way a result of Solecki regarding a second class of monoids connected to the Furstenberg-Katznelson  Ramsey Theorem. The results obtained suggest a possible connection with Schützenberger's theorem and finite automata theory.}

	\tableofcontents

\vfill \hrule
	
	    \msc{\textit{Primary:  03C98, 05D10. Secondary: 03H99, 20M32}}

$^\dag$ The research of both authors was supported in part by the Universit\`a degli Studi di Torino (Italy), in part by the Gruppo Nazionale per le Strutture Algebriche, Geometriche e le loro Applicazioni (GNSAGA) of the Istituto Nazionale di Alta Matematica (INdAM) of Italy, and in part by project PRIN 2017 "Mathematical Logic: models, sets, computability", prot. 2017NWTM8R. The research of the first author was supported in part by the Research Institute for Mathematical Sciences,
an International Joint Usage/Research Center located in Kyoto University.

\section{Introduction}\label{section:intro}


	One of the most celebrated theorems in Ramsey theory is due to Hindman \cite{MR349574}. It states that for every semigroup $(S,\cdot)$, for every finite coloring of $S$ there is an 
	infinite sequence $\bar{s}= (s_i)_{i\in\omega} \in S^\omega$ such that the following set is monochromatic \[ \emph{$\fp(\bar{s})$}=\{ s_{i_0}\cdot\dots\cdot s_{i_n} : n\in\omega, i_0< \dots < i_n \}.\]
	
	A natural question is whether a theorem of this kind can be proved if we allow the elements of the semigroup to be moved by some action. 

The first answers to this question were given by Carlson \cite{MR926120} and Gowers \cite{MR1164759}  who studied actions of specific monoids. They used analogues of the following notion:

\begin{definition}
Let $M$ be a monoid acting by endomorphisms on a partial semigroup $S$, and let $\bar{s}$ be a sequence of elements of $S$. The (combinatorial) \emph{$M$-span} of $\bar{s}$ is the set \[\emph{$\langle\bar{s}\rangle_{M}$} =\big\{m_{0}\,{s_{i_0}}\cdots m_{n}\,{s_{i_n}}
:  n\in\omega, 
i_0<\dots<i_n, m_i\in M, 
\textit{ at least one } m_i\textit{ is } 1_M
\big\}.\]
\end{definition}

Solecki realized that these theorems share the same underlying structure, from which he isolated the notion of Ramsey monoid.  
We report here the original definition as stated in \cite{MR3904213}. We refer the reader to  Proposition~\ref{prop:transfer lemma for semigroups} for other equivalent definitions. 

Let $M$ be a monoid and $X$ be a set. If $M$ acts on $X$, we say that $X$ is an $M$-set.
Suppose $(X_n)_{n\in\omega}$ is a family of $M$-sets.
We say that the action of $M$ on $(X_n)_{n\in\omega}$ is \emph{uniform} if for every $k,n\in\omega$, $m\in M$ and $x \in  X_k\cap X_n$, if $m_kx$ and $m_nx$ are the results of the action of $m$ on $x$ respectively in $X_k$ and in $X_n$, then  $m_kx=m_nx$. 
Notice that the action of $M$ on $(X_n)_{n\in\omega}$ is uniform if and only if it extends to $\bigcup_{n\in\omega} X_n$.
Let $X=\bigcup_{n\in\omega}X_n$. Define $\emph{$W_X$}=(X^{<\omega},\conc)$
to be the free semigroup on the alphabet $X$, with $\conc$ the concatenation of sequences.
Define $\emph{$\langle(X_n)_{n\in\omega}\rangle$}$
to be the partial subsemigroup of $W_X$ consisting of all sequences $x_1\conc ... \conc x_n\in W_X$ for which there exists $i_1<...<i_n\in \omega$ such that $x_k\in X_{i_k}$.
If the action of $M$ is uniform on $(X_n)_{n\in\omega}$, it is possible to define the \emph{coordinate-wise action} of $M$ on $\langle(X_n)_{n\in\omega}\rangle$ by setting $ m(x_1\conc ... \conc x_n)=m(x_1)\conc ... \conc m(x_n)$. This is an action by endomorphisms. 
An infinite sequence $\bar{s}\in (\langle(X_n)_{n\in\omega}\rangle)^\omega$ is said \emph{basic} if the product $s_{i_0}\conc{\dots}\conc {s_{i_n}}$ is in $\langle(X_n)_{n\in\omega}\rangle$  for every $i_0<\dots<i_n$. This implies that if $M$ acts uniformly on $(X_n)_{n\in\omega}$ and $\bar{s}$ is basic, then every product in $\langle \bar{s}\rangle_M$ is defined. 
A \emph{pointed $M$-set} is an $M$-set $X$ together with  
a \emph{distinguished point} $x\in X$ such that $Mx=\{mx:m\in M\}=X$.

\begin{definition}\label{def:Ramsey_original_intro}
A monoid $M$ is said \emph{Ramsey} if for all sequences of pointed $M$-sets $(X_n)_{n\in\omega}$ on which $M$ acts uniformly  and for all finite colorings of $\langle (X_n)_{n\in\omega} \rangle$ there is a basic sequence $\bar{s}\in (\langle (X_n)_{n\in\omega} \rangle)^{\omega}$ such that $s_n$ has a distinguished point for every $n\in\omega$ and the span $\langle \bar{s}\rangle_{M}$ 
is monochromatic.
\end{definition}

The notion of Ramsey monoid provides a common framework for many theorems in combinatorics.
For example, Hindman's theorem can be restated as "\textit{The trivial monoid $\{1\}$ is Ramsey}". 
Similarly, Carlson's theorem and Gowers' $\FIN_k$ theorem can be seen just as two examples of Ramsey monoids (see also Proposition~\ref{prop:transfer lemma for semigroups}). 
Hence, every new Ramsey monoid gives a new generalization of Hindman's theorem as powerful as Carlson's and Gowers' theorems. 

Furthermore, from Carlson's and Gowers' theorems one can get examples of Ramsey spaces, see \cite[Section 4.4]{MR2603812}.
In the same way,  one can show that any new example of Ramsey monoid gives new examples of Ramsey spaces.

 Let us recall some notation from \cite{MR3904213}. 
Given a monoid $M$, define $\emph{$\XX(M)$}=\{aM: a\in M \}$.
	We say that $\XX(M)$ is linear if it is linearly ordered by inclusion. Let \emph{$\R$} be the equivalence relation on $M$ defined by $a\Rbin b$ if $aM=bM$. A monoid $M$ is called \emph{almost ${\R}$-trivial} if for every $\R$-class $[a]_\R$ with more than one element we have $Ma=\{a\}$. 
	
To the best of our knowledge, the only known examples of Ramsey monoids before Solecki's paper were given by Carlson's and Gowers' theorems. 
Solecki in \cite[Corollary 4.5]{MR3904213} proved that the class of finite almost ${\R}$-trivial monoids with linear $\XX(M)$, that includes Carlson's and Gowers' monoids, is Ramsey. Also, he proved that Ramsey monoids have a linear $\XX(M)$.

 We work here with a well-known class of monoids that extends the one of almost $\R$-trivial monoids. 
A monoid is said \emph{ aperiodic} if for every $a\in M$ there exists $n\in\omega$ such that $a^n=a^{n+1}$ (see \cite{pin2010mathematical}).
The first main achievement of this paper is an improvement of Solecki's result:  first, we extend \cite[Corollary 4.5]{MR3904213} to the wider class of finite aperiodic monoids with linear $\XX(M)$, and  secondly, we prove that aperiodicity is also a necessary condition for being Ramsey, giving thus a complete characterization of finite Ramsey monoids.

\begin{Main 1}\label{main theorem introduction}
A finite monoid $M$ is Ramsey if and only if it is aperiodic and $\XX(M)$ is linear.
\end{Main 1}

To introduce the other peak of this paper we need some more notions from \cite{MR3904213}.
	
Given a monoid $M$, \emph{$\YY(M)$}$\subseteq \powerset (\XX(M))$ consists of the non-empty subsets of $\XX(M)$ which are linearly ordered by inclusion. Given $x, y\in \YY(M)$, define \emph{$x\leq_{\YY(M)}y$} if $x\subseteq y$ and all elements of $y\setminus x$ are larger with respect to $\subseteq$ than all elements of $x$.

	Let $\langle \YY(M)\rangle$, with operation $\vee$, be the semigroup freely generated by $\YY(M)$  modulo the relations
	\[ p\vee q =q =q\vee p \mbox{ for }p\leq_{\YY(M)} q. \]
	
	We say that $M$ is \emph{$\YY$-controllable} if	for every finite $F\subseteq \langle \YY(M)\rangle$, for every $\textbf{y}$ maximal element in $ \YY(M)$, for every sequence of pointed $M$-sets $(X_n)_{n\in\omega}$ on which $M$ acts uniformly and for every finite coloring of $\langle (X_n)_{n\in\omega} \rangle$ there is a basic sequence $\bar{s}\in (\langle (X_n)_{n\in\omega} \rangle)^{\omega}$ such that $s_n$ has a distinguished point for every $n\in\omega$ and such that for every $m, n\in\omega$ and for every $a_i, b_j\in M$  \underline{if} $ a_0\textbf{y}\vee \dots \vee a_n\textbf{y}\in F$ and
	$ a_0\textbf{y}\vee \dots \vee a_n\textbf{y}=b_0\textbf{y}\vee\dots\vee b_m \textbf{y}$,
	\underline{then} $ a_0s_{i_0}\cdot {\dots} \cdot a_ns_{i_n} $
	has the same color of $b_0s_{j_0}\cdot {\dots} \cdot b_ms_{j_m}$, for every $i_0< \dots <i_n$,  $j_0<\dots <j_m $.

	Another major result of Solecki \cite[Corollary 4.3]{MR3904213} is that almost $\R$-trivial monoids are $\YY$-controllable. This has amongst its consequences a  theorem of Furstenberg and Katznelson \cite{MR1039473}. We refer the reader to \cite[Section 4]{MR3904213} for a detailed discussion about this connection. We extend Solecki's result to a larger class of monoids and we prove that aperiodicity is a necessary condition for being $\YY$-controllable. 
	
	Given a monoid $M$, define \emph{$\XX_\R(M)$}$=\{aM: [a]_\R \mbox{ has more than one element}\}$.
	We say that $\XX_\R(M)$ is linear if it is linearly ordered by inclusion.
	\begin{Main 2}\label{main introduction 2, Y(M) controllable}
		Let $M$ be a finite monoid. If $M$ is aperiodic and $\XX_\R(M)$ is linear, then $M$ is $\YY$-controllable. If $M$ is $\YY$-controllable, then it is aperiodic.
	\end{Main 2}

	The proof of Theorem~\ref{main introduction 2, Y(M) controllable} is first presented as divided into two parts: a new result about monoid actions on compact topological right topological semigroups and a reformulation of known results by Solecki (the latter are then re-proved in Section~\ref{section:proof of main theorem with model theory} using model theory).
	Namely, the main technical novelty of this paper, which allows us to prove one direction of Theorem~\ref{main introduction 2, Y(M) controllable} and consequently one direction of Theorem~\ref{main theorem introduction}, is the following result. Relevant notions are defined in Section~\ref{section: compact r.t.s. part I}.

	\begin{theorem}\label{theorem monoid actions introduction} 
		Let $M$ be a finite aperiodic monoid.
		Let $U$ be a compact right topological semigroup on which $M$ acts by continuous endomorphisms.
		If $\XX_\R(M)$ is linear, then there exists a minimal idempotent $ {u\in U}$ such that $a( {u})=b( {u})$ for all couples $a,b\in M$ such that $a\Rbin b$.
	\end{theorem}
	
The notion of aperiodic monoid plays a central role in both our main results.
This class of monoids is also involved in one of the most important theorems in finite automata theory, also dealing with the semigroup of words, due to Schützenberger \cite{MR176883}.
This suggests there might be a possible connection between automata theory and Ramsey theory.

From now on, we review the structure of the paper section by section.

Section~\ref{section:equivalent_classes} is introductory to the theory of monoids and Green's relations. 
First, we recall some basic properties of the class of aperiodic monoids and of  the class of monoids with linear $\XX(M)$. This is meant to provide alternative necessary or sufficient conditions for a monoid to be Ramsey, using Theorem~\ref{main theorem introduction}. 
Then, we introduce the class of aperiodic monoids with linear $\XX_\R(M)$, and show that this class properly extends both the one of almost $\R$-trivial monoids introduced by Solecki in \cite{MR3904213} and that of aperiodic monoids with linear $\XX(M)$. 

In Section~\ref{section: compact r.t.s. part I}, we study actions of aperiodic monoids with linear $\XX_\R(M)$ on compact right topological semigroups, proving Theorem~\ref{theorem monoid actions introduction}. This theorem and Corollary~\ref{cor:action_of_R_rigid_monoid_fixes_semigroups_bis}  seem to show that aperiodicity plays a relevant role in dynamic theory. 
We complete the proofs of Theorems~\ref{main theorem introduction} and \ref{main introduction 2, Y(M) controllable} in two different ways in  Section~\ref{section: dynamic for furstenberg and katznelson} and Section~\ref{section:proof of main theorem with model theory}.

In Section~\ref{section: dynamic for furstenberg and katznelson} we guide the reader to a rephrasing of Solecki's work, to show how one implication of Theorem~\ref{main theorem introduction} and one of Theorem~\ref{main introduction 2, Y(M) controllable} follow from Theorem~\ref{theorem monoid actions introduction}. This is done in two steps. First, we use a lemma of Solecki {\cite[Lemma 2.5]{MR3904213}}.  Secondly, we guide the reader through the ultrafilter proof of Solecki.
In Theorem~\ref{main 2} we prove the remaining implication of Theorem~\ref{main introduction 2, Y(M) controllable}, i.e. that every $\YY$-controllable monoid is aperiodic. The same argument shows that Ramsey monoids are aperiodic.

In section \ref{section:proof of main theorem with model theory}, we give an alternative to the ultrafilter proof of Solecki, using model theory. Here we use the space of types where Solecki uses the space of ultrafilters. 
Model theory is limited to Section~\ref{section:proof of main theorem with model theory}, Theorem~\ref{Milliken-Taylor}, and Proposition~\ref{example YY controllable monoid not linear} and just basic notions are assumed.

In Section~\ref{section:Equivalent definitions of Ramsey monoid}, first we explain in detail the relation between Ramsey monoids and Carlson's and Gowers' theorems. We show that the definition of Ramsey monoid presented so far can be reformulated in many equivalent ways, and in particular one can choose semigroups other than $\langle (X_n)_{n\in\omega}\rangle$. 
Secondly, we point out that Ramsey monoids satisfy stronger properties than being Ramsey.
In particular, in Theorem~\ref{Milliken-Taylor} we prove a common generalization of Milliken-Taylor theorem \cite{MR373906}, \cite{MR424571}  and Theorem~\ref{main theorem introduction}, combining Ramsey's theorem and the definition of Ramsey monoid.

In Section~\ref{section: open questions}, we provide further examples of $\YY$-controllable monoids, through Proposition~\ref{example YY controllable monoid not linear}. This shows that there are $\YY$-controllable monoids such that $\XX_\R(M)$ is not linear. Finally, we collect some open questions in the area.

\section{Aperiodic monoids, $\XX(M)$ and $\XX_\R(M)$}\label{section:equivalent_classes}

In this section, we introduce the basic notions and definitions about monoids we are going to use throughout the paper.

One of the best ways to describe monoids and semigroups is using Green's relations. They were first introduced by Green in his doctoral thesis and in \cite{MR42380}. 
The Green's relations \emph{$\R$}, \emph{$\L$} and \emph{$\J$}  on a monoid $M$ are the equivalence relations defined by, respectively, $a\Rbin b$ if $aM=bM$, $a\Lbin b$ if $Ma=Mb$ and $a\Jbin b$ if $MaM=MbM$. 
The Green's relation \emph{$\H$} is the intersection of $\R$ and $\L$, while the Green's relation \emph{$\D$} is the smallest equivalence relation containing both $\L$ and $\R$. In every finite monoid, we have $\D=\J$. 
The Green's relations induce quasi-orders on the monoid. Given two element $a,b\in M$, define $a\leq_\R b$ if $aM\subset bM$, $a\leq_\L b$ if $Ma\subset Mb$, $a\leq_\J b$ if $MaM\subset MbM$, and finally $a\leq_\H b$ if both $a\leq_\R b$ and $a\leq_\L b$ hold. 
If $\K$ is an equivalence relation, we say that an equivalence class $[a]_\K$ is \emph{trivial} if it contains exactly one element, and we say that a monoid $M$ is \emph{${\K}$-trivial} if every $\K$-class is trivial. 
For more information about Green's relations, see e.g. \cite{CliffordMR0132791}.

A monoid is said \emph{aperiodic} if for all $a\in M$ there exists $n\in\omega$ such that $a^n=a^{n+1}$. 
The class of finite aperiodic monoids has been widely studied,
as it is involved in one of the most important theorems in finite automata theory,  due to Schützenberger \cite{MR176883}.  
It states that star-free languages are exactly those languages whose syntactic monoid is finite and aperiodic. By a result of McNaughton and Papert, these also correspond to the languages definable in FO$[<]$, i.e. first-order logic with signature ${<}$ \cite{MR0371538}.

Among finite monoids, the class of aperiodic monoids can be characterized in many ways. 
We report here some of the most famous  options used in literature. 
Among all possibilities, we isolate the notion of $\R$-rigid monoid as the operative definition we are going to use in the proofs of the next section. 

\begin{definition}
	A monoid is said \emph{ $\R$-rigid} if for every $a,b\in M$, if $ab\Rbin b$, then $ab=b$.
\end{definition}

\begin{proposition}\label{prop:equivalent conditions for aperiodic monoids}
	Let $M$ be a finite monoid. The following are equivalent:
	\begin{enumerate}
		\item $M$ is aperiodic.\label{R_rigid_3}
		\item For every $g,a,g'\in M$, if $gag'= a$, then $ga=ag'=a$.\label{R_rigid_4}
		\item $M$ is $\R$-rigid. \label{R_rigid_1}	
		\item $M$ is $\H$-trivial.\label{R_rigid_5} 
        \item $M$ contains no non-trivial subgroup.\label{R_rigid_6}
	\end{enumerate}
\end{proposition}

For the ease of the reader, we report also a short proof of the equivalence.

\begin{proof}	
	First, assume \ref{R_rigid_3}, and let $g,a,g'\in M$ be such that $gag'=a$. 
	By induction this implies $g^n a (g')^n=a$ for every $n\in\omega$. Choose $n$ such that $g^{n+1}=g^n$ and $(g')^{n+1}=(g')^n$. Then, \ref{R_rigid_4} holds since \[
	ga=g(g^n a (g')^n)=g^{n+1} a (g')^n=g^n a (g')^n=a
	\] 
	and similarly, 
	\[
	ag'=(g^n a (g')^n)g'=g^{n} a (g')^{n+1}=g^n a (g')^n=a.
	\]
	
	If $ab\Rbin b$, then by definition of $\R$ there exists $a'\in M$ such that $b=aba'$, hence \ref{R_rigid_4} implies \ref{R_rigid_1}.
	
	Notice that if $a,b\in M$ are such that $a\Hbin b$, then in particular $a\Rbin b$ and there is $x\in M$ such that $xb=a$, hence \ref{R_rigid_1} implies \ref{R_rigid_5}.
	
	Now if $G\subset M$ is a subgroup of $M$, then for every $a,b\in G$ there are $x,y$ such that $ax=ya=b$, and symmetrically there are $x',y'$ such that $bx'=y'b=a$, hence $a\Hbin b$ and $G$ is contained inside one single $\H$-class. Therefore, \ref{R_rigid_5} implies \ref{R_rigid_6}.

	Finally, notice that if $M$ is finite, then for every $a\in M$ there are minimal $n,k\in \omega$ such that $a^{n+k}=a^{n}$. Then, the set $\{a^{n+i}:  i<k\}$ is a subgroup of $M$, and it is trivial if and only if $k=1$, hence \ref{R_rigid_6} implies \ref{R_rigid_3}.
\end{proof}

The class of aperiodic monoids is closed under most basic operations. For example, the following holds:

\begin{proposition}\label{prop:operation_aperiodic}
Let $(S_1,\ast_1),\dots, (S_n,\ast_n)$ be aperiodic semigroups. Then, the following are aperiodic:
\begin{enumerate}
    \item The product monoid $S_1\times \dots \times S_n$ with coordinate-wise operation.
    \item The disjoint union $S_1\sqcup \dots \sqcup S_n$ with operation $a\ast b=a\ast_i b$ when $a,b\in S_i$, and $a\ast b= b\ast a = a$ if $a\in S_i$ and $b\in S_j$ with $i<j$.\label{prop:operation_aperiodic-3}
\end{enumerate}
\end{proposition}

For more information about aperiodic monoids and their relations with languages and automata, see for example  \cite{lawson2003finite} or \cite{pin2010mathematical}.

Let us move to the next class. Given a monoid $M$, define $\emph{$\XX(M)$}=\{aM: a\in M \}$.
We say that $\XX(M)$ is linear if it is linearly ordered by inclusion (equivalently, if $\leq_\R$ is a total quasi-order).
The existence of a total quasi-order affects the behaviour of Green's relations, and having that $\leq_\R$ is total has even stronger consequences.
The next proposition collects some well-known properties of monoids where $\leq_\R$ is total  (see \cite[Proposition 3.18-3.20]{HofmannMostertMR0209387}).

\begin{proposition}\label{prop:Green_relations_for_XM_linear}
    Let $M$ be a finite monoid with linear $\XX(M)$. Then, the following hold:
    \begin{enumerate}
    \item For every $a\in M$, the principal right ideal $aM$ is a both-sided ideal.
    \item $\J=\D=\R$ and $\L=\H$, while ${\leq_\R}={\leq_\J}$ and ${\leq_\L}={\leq_\H}$.\label{prop:Green_relations_for_XM_linear-2}
    \item $\R$ is a congruence relation.\label{prop:Green_relations_for_XM_linear-3}
    \item $\leq_\R$ is translation-invariant on both sides.\label{prop:Green_relations_for_XM_linear-4}
    \end{enumerate}
\end{proposition}

For more information about monoids with linear $\XX(M)$, see for example \cite{HofmannMostertMR0209387}.

 Combining results about aperiodic monoids with results about monoids with linear $\XX(M)$, one can obtain further properties and characterizations of the class of finite aperiodic monoid with linear $\XX(M)$. For example, a finite  monoid with linear $\XX(M)$ is aperiodic if and only if it is $\L$-trivial.
Notice that in light of Theorem~\ref{main theorem introduction}, 
every property of this class of monoids will give a necessary condition for a monoid to be Ramsey.

Finally, let us introduce a seemingly new class of monoids, the class of aperiodic monoids with linear $\XX_\R(M)$. It is one of the key notions for the other main result of this paper, Theorem~\ref{main introduction 2, Y(M) controllable}. 

Let {$\XX_\R(M)$} be the subset of $\XX(M)$ of those $aM$ such that $[a]_\R$ is non-trivial.
We say that $\XX_\R(M)$ is linear if it is linearly ordered by inclusion. 
Recall  also that $M$ is called \emph{almost ${\R}$-trivial} if for every non-trivial $\R$-class $[a]_\R$ we have $Ma=\{a\}$ (see \cite{MR3904213} and \cite{khatchatourian2018almost}).

\begin{proposition}\label{Almost R-trivial monoid are aperiodic}
Every finite almost $\R$-trivial monoid $M$ is aperiodic and has linear $\XX_\R(M)$. 
\end{proposition}

\begin{proof}
Let $M$ be a finite almost $\R$-trivial monoid.
First, we want to show that $\XX_\R(M)$ has at most one element that is the minimum of $\XX(M)$ (and so $\XX_\R(M)$ is in particular linearly ordered by inclusion).
If  $[a]_\R$  is a non-trivial $\R$-class  then, for every $m\in M$ we have $ma=a$, that means $a\in mM$ and $aM\subseteq mM$. Hence, if $[a]_\R$ and $[b]_\R$ are non-trivial $\R$-classes, then we have $aM= bM$. 
Now let us prove that $M$ is aperiodic. Since $M$ almost $\R$-trivial, then $Mb=\{b\}$ holds for every non-trivial $\R$-class. This in particular implies that $(Mb)\cap [b]_\R=\{b\}$ holds for every $\R$-class, and this is just a rephrasing of the $\R$-rigid condition. Then the claim follows from Proposition~\ref{prop:equivalent conditions for aperiodic monoids}. 
\end{proof}

Notice that the converse does not hold, as 
it is easy to show that there are aperiodic monoids with linear $\XX_\R(M)$ that have more than one non-trivial $\R$-class (for example, by combining almost $\R$-trivial monoids with point~\ref{prop:operation_aperiodic-3} of Proposition~\ref{prop:operation_aperiodic}; see also Example~\ref{example Carlson product Gowers}). 
Also, there are aperiodic monoids that have non-trivial $\R$-classes $[a]_\R$ such that $a$ is not idempotent  (a minimal example is given by the monoid in Table~\ref{tab:example_monoid_1}, see also Example~\ref{example Carlson product Gowers}). 
These conditions are impossible for almost $\R$-trivial monoids, as shown in the proof of Proposition~\ref{Almost R-trivial monoid are aperiodic}.
Finally, there are examples of aperiodic monoids with linear $\XX_\R(M)$ that do not have linear $\XX(M)$ (the easiest examples coming from $\R$-trivial monoids).
Thus, the class of aperiodic monoids with linear $\XX_\R(M)$ properly extends both the class of almost $\R$-trivial monoids and the class of aperiodic monoids with linear $\XX(M)$.

\begin{example}\label{example Carlson product Gowers}
Consider the Gowers' monoid $G_k=(\{0,\dots, k-1\}, \bar{+})$ with operation $i\mathbin{\bar{+}} j =\min(i+j,k-1)$. Consider also the Carlson's semigroup ${C}_A=(A,\ast)$, i.e. a finite set $A$ with operation $a\ast b=b$ for every $a,b\in A$. Let ${C}^1_A={C}_A\cup \{1_{C^1_A}\}$ be the corresponding monoid. 
  Then, for every $k$ and $A$ the monoid $M=(G_k\times C^1_A)$ is aperiodic and has linear $\XX_\R(M)$, while $\tilde{M}=(G_k\times {C}_A)\cup\{1_{\tilde{M}}\}$ is aperiodic, has linear $\XX(\tilde{M})$ and all its $\R$-classes other than $[1_{\tilde{M}}]_\R$ are non-trivial. If $k\geq 2$, neither of these monoids is almost $\R$-trivial.
\end{example}

For those familiar with finite automata theory, Schützenberger Theorem provides a wonderful way to produce examples of aperiodic monoids. Starting from a star-free language $S$, or from a formula in FO$[<]$, we always generate a finite aperiodic syntactic monoid. 
For example, the monoid from Table~\ref{tab:example_monoid_1} is the syntactic monoid of the star-free language $S$ in the alphabet $A=\{a,g,h\}$ defined as 
\[S=\{g,h\}^\ast h \cup \{g,h\}^\ast a \{g,h\}^\ast g\cup A^\ast a A^\ast a A^\ast\]
or, equivalently, defined by the formula in FO$[<]$ that says \textit{``the word is non-empty, and if there are no letters $a$, then the word ends with $h$, and if there is exactly one letter $a$, then the word ends with $g$''}.

\begin{table}[h]
        \centering
        \begin{tabular}{c|c c c c c}
             {1} & $0$ & $a$ & $b$ & $g$ & $h$ \\
             \hline
             $0$ & $0$ & $0$ & $0$ & $0$ & $0$ \\
             $a$ & $0$ & $0$ & $0$ & $b$ & $a$ \\
             $b$ & $0$ & $0$ & $0$ & $b$ & $a$ \\
             $g$ & $0$ & $a$ & $b$ & $g$ & $h$ \\
             $h$ & $0$ & $a$ & $b$ & $g$ & $h$
        \end{tabular}
        \caption{Syntactic monoid of the language $S$.
        }
        \label{tab:example_monoid_1}
    \end{table}

 \section{Dynamic theory}\label{section: compact r.t.s. part I}

In this section, we study actions of aperiodic monoids with linear $\XX_\R(M)$ on compact right topological semigroups.  The main objective is to prove Theorem~\ref{thm:get idempotent for main theorem}, i.e. Theorem~\ref{theorem monoid actions introduction} from the introduction. This result reveals the relation between aperiodic monoids and dynamic theory and it will be the key point to prove one direction of Theorem~\ref{main theorem introduction} and one direction of Theorem~\ref{main introduction 2, Y(M) controllable}.
The advantage to work with compact right topological semigroups is that they are the common ground for many different thecniques, either from logic or ergodic theory (see e.g. \cite{MR1262304}, \cite{MR1039473}, \cite{MR3896106}, \cite{MR3904213},  \cite{MR2603812}). 

Let us recall some notions.
A semigroup $(U,\cdot)$ with a topology $\tau$ is a {right topological semigroup} if the map $x\mapsto x\cdot u$ is continuous from $U$ to $U$ for every $u\in U$. It is called compact if $\tau$ is compact.
	A element $u$ in a semigroup $(U,\cdot)$ is called idempotent if $u\cdot u=u$. The set of idempotents of $U$ is denoted by \emph{$E(U)$}. We define a partial order \emph{$\leq_U$} in $E(U)$ by \[u\leq_U v \Longleftrightarrow uv=u=vu.\]
	Finally, let \emph{$I(U)$} be the smallest compact both-sided ideal of $U$. It exists by compactness of $U$.

We report some facts about idempotents, corresponding to \cite[Lemma 2.1, Lemma 2.3, Corollary 2.5]{MR2603812}. 

\begin{proposition}\label{facts idempotents}
	Let $U$ be a compact right topological semigroup. Then,\begin{enumerate}
	\item $E(U)$ is non-empty.  \label{facts idempotents-1}
	\item For every idempotent $v$ there is a $\leq_U$-minimal idempotent $u$ such that $u\leq_U v$. \label{facts idempotents-2}
	\item Any both-sided ideal of $U$ contains all the minimal idempotents of $U$. \label{facts idempotents-3}
\end{enumerate}  
\end{proposition}

\begin{fact}\label{fact a(M)=b(M) ---> a(U)=b(U)}
    Let $M$ be a monoid, let $U$ be any set, and fix a left action of $M$ on $U$. 
    Then, for every $a,b\in M$ such that $aM\subseteq bM$ we have $a(U)\subseteq b(U)$.
    \end{fact}
    \begin{proof}
    In fact, if $bm=a$ for some $m\in M$, then $a(U)= b(m(U))\subseteq b(U)$.
\end{proof}

 In particular, if $a\Rbin b$, then $a(U)=b(U)$.

\begin{lemma}\label{lem:action_of_R_rigid_monoid_fixes_semigroups_bis}
    Let $M$ be a finite aperiodic monoid such that $\XX_\R(M)$ is linear. Then, for every distinct $a,b\in M$ with $a\Rbin b$ there are two distinct $g, h\in M$ such that $ag=b$, $bh=a$ and $gh=h$, $hg=g$.
    This in particular implies $gM=hM$.
\end{lemma}

\begin{proof}
Fix a non-trivial $\R$-class $[c]_\R$ and let $a,b\in [c]_\R$ with $a\neq b$. 
For every $y,z\in M$, define \[G_{y,z}=\{g_{y,z}\in M:  yg_{y,z}=z\}.\] 
Notice that if $y\Rbin z$, then $G_{y,z}$ is non-empty.
 Let $g\in G_{a,b}$ be such that $gM$ is minimal in $\{xM:  x\in G_{a,b}\}$, and similarly let $h\in G_{b,a}$ be such that $hM$ is minimal in $\{xM:  x\in G_{b,a}\}$.

We claim that $g\Rbin h$.  
   Notice that $hgh\in G_{b,a}$ since $bhgh=agh=bh=a$. Since $hghM\subset  hgM\subset  hM$, by minimality of $hM$ we have $hghM=hgM=hM$, so $h \Rbin hg$ and $h\Rbin hgh$. Notice that $hg\in G_{b,b}$ and that $G_{b,a}\cap G_{b,b}=\emptyset$, so $h\neq hg$ and the class $[h]_\R$ is non-trivial. Similarly, $g\Rbin gh \Rbin ghg$, $gh\in G_{a,a}$ and so the class $[g]_\R$ is non-trivial. 
By hypothesis this implies either $gM\subset hM$ or $hM\subset gM$.  
Suppose for example $gM\subset hM$. Then, 
\[|h(gM)|\leq |gM|\leq |hM|=|hgM|\] 
and so $|hM|=|gM|$ and $hM=gM$. 
This implies that $h$, $hg$,  $g$, $gh$,  are all in the same $\R$-class, hence $gh=h$ and $hg=g$, by definition of $\R$-rigid  and Proposition~\ref{prop:equivalent conditions for aperiodic monoids}.
\end{proof} 

\begin{lemma}\label{lem:XrM linear implies product of classes stays into classes}
    Let $M$ be a finite aperiodic monoid such that $\XX_\R(M)$ is linear. Then, for every $a\in M$, if there are $b,c\in [a]_\R$ such that $bc=c$, then for every $b,c\in [a]_\R$ we have $bc=c$.
\end{lemma}

\begin{proof}
First, notice that if $xy=y$ for some $x,y\in M$, then $xz=z$ for every $z\in [y]_\R$ since $xzM=xyM=yM=zM$ and since $M$ is $\R$-rigid  by Proposition~\ref{prop:equivalent conditions for aperiodic monoids}. 

Hence, we just need to prove that given a non-trivial $\R$-class $[a]_\R$ such that $ax=x$ for every $x\in [a]_\R$, and given an element $b\in [a]_\R$ with $b\neq a$, then we have $ba=a$. 

Let $g,h$ be such that $ag=b$ and $bh=a$. 
Notice that $ha\Rbin hb$ since $haM=hbM$, and also $ha\neq hb$ since $bha=a\neq b=bhb$. Then $haM\in \XX_\R(M)$ and so $haM\subset aM$ or $aM\subset haM$. We have $|aM|=|bhaM|\leq |haM|\leq |aM|$, hence $|haM|=|aM|$ and by linearity of $\XX_\R(M)$ we must have $haM=aM$. Since $M$ is $R$-rigid, $ha=a$ holds and we have  $ba=bha=a$.
\end{proof}

With this, we are ready to prove Theorem~\ref{theorem monoid actions introduction}.

\begin{theorem}\label{thm:get idempotent for main theorem} 
Let $M$ be a finite aperiodic monoid.
Let $U$ be a compact right topological semigroup on which $M$ acts by continuous endomorphisms.
If $\XX_\R(M)$ is linear, then there exists a minimal idempotent $ {u}\in E(U)\cap I(U)$ such that $a( {u})=b( {u})$ for all couples $a,b\in M$ such that $a\Rbin b$.
  \end{theorem}

\begin{proof}
Let $a_0 M \subsetneq ...\subsetneq a_n M$ be an increasing enumeration of $\XX_\R(M)$ and define $a_{n+1}=1$.
Every $a_i(U)$ is a semigroup, since $a_i(u_1)\cdot a_i(u_2) =a_i(u_1\cdot u_2)$, and it is compact because it is a continuous image of a compact space. Then, $a_i(U)$ is a compact subsemigroup of the compact semigroup $a_{i+1}(U)$. 
We want to find a chain of idempotents $u_i$ such that $u_{i+1}\leq_U u_i$ and such that $u_i$ is minimal in $E(a_i(U))$ with respect to $\leq_{a_i(U)}$, for every $i\leq n+1$.

First, by points~\ref{facts idempotents-1} and~\ref{facts idempotents-2} of Proposition~\ref{facts idempotents}, we can find $u_0\in a_0(U)$ satisfying the requirement. 
Then, suppose we have $u_i \in a_i(U)$ idempotent. Since $a_i(U)\subset a_{i+1}(U)$ we may apply point~\ref{facts idempotents-2} of Proposition~\ref{facts idempotents} to find $u_{i+1}\in a_{i+1}(U)$ idempotent such that $u_{i+1}\leq_{a_{i+1}(U)} u_i$ and $u_{i+1}$ is minimal in $E(a_{i+1}(U))$, and this concludes the construction. Since $a_{n+1}=1$ and $E(a_{n+1}(U))=E(U)$, by point \ref{facts idempotents-3} of Proposition~\ref{facts idempotents} we also know that $u_{n+1}\in I(U)$.

We claim that $u=u_{n+1}$ satisfies the requirements of the thesis. 

First, we want to show that for each $\R$-class $[a_i]_\R$ with $a_ia_i=a_i$
we have \[b( u)= u_i \text{ for all } b\in [ a_i]_\R.\]

By Lemma~\ref{lem:XrM linear implies product of classes stays into classes}, for every $b\in[a_i]_\R$ we have $ba_i=a_i$, and this implies that for every $v\in a_i(U)$, say $v=a_i(u_v)$, we have $b(v)=b(a_i(u_v))=a_i(u_v)=v$. In particular for $v=u_i$, we have $b(u_i)=u_i$.
Notice that the action of $M$ is order preserving on $(U,\leq_U)$, since it is by endomorphisms. 
Since $ u\leq_U u_i$ we get 
\[b( u)\leq_U b( u_i)= u_i.\] 
Thus, $b( u)\leq_{ a_i(U)} u_i$, and since $ u_i$ is minimal in $ a_i(U)$, we get $b( u)= u_i$.

Now consider a non trivial $\R$-class $[a_i]_\R$ such that $a_ia_i\notin [a_i]_\R$, and let $a,b\in [a_i]_\R$. Let $g,h$ be given as in  Lemma~\ref{lem:action_of_R_rigid_monoid_fixes_semigroups_bis} such that $ag=b$ and $bh=a$ and $hg=g$. Notice that this implies $bg=bhg=ag=b$ and also $h(u)=g(u)$, since $[g]_\R$ belongs to the previous case. Then, $a(u)=bh(u)=bg(u)=b(u)$.
\end{proof}

We take the opportunity to state a corollary of Lemma~\ref{lem:action_of_R_rigid_monoid_fixes_semigroups_bis}.

\begin{corollary}\label{cor:action_of_R_rigid_monoid_fixes_semigroups_bis}
	Let $M$ be a finite aperiodic monoid such that $\XX_\R(M)$ is linear, let $U$ be a set and fix a left action of $M$ on $U$. Then, for every $a,b\in M$ with $a\Rbin b$ and for every $u\in a(U)$ we have $a(u)=b(u)$.
\end{corollary}

\begin{proof}
	Let $a,b\in M$ be such that $a\Rbin b$ and $a\neq b$, and let $g,h\in M$ be given by Lemma~\ref{lem:action_of_R_rigid_monoid_fixes_semigroups_bis} such that $ag=b$ and $bh=a$, and $gh=h$ and $hg=g$. This in particular implies $gg=ghg=hg=g$, and $bg=bhg=ag=b$. Notice that by linearity of $\XX_\R(M)$ either $a(M)\subset g(M)$ or $g(M)\subseteq a(M)$ holds, since both $[a]_\R$ and $[g]_\R$ are non-trivial. Then, we have $a(M)\subseteq g(M)$, since $|aM|=|bhM|\leq |hM|=|gM|$, and also $a(U)\subseteq g(U)$, by Fact~\ref{fact a(M)=b(M) ---> a(U)=b(U)}.
	Fix $u\in a(U)$ and find $v\in U$ such that $u=g(v)$. We have \par \vspace{\abovedisplayshortskip}\hfill$\displaystyle a(u)=a(g(v))=a((gg)(v)))=ag(g(v)))=b(g(v))=b(u).$ 
\end{proof}

\section{Coloring theorems and aperiodic monoids}\label{section: dynamic for furstenberg and katznelson}

In this section, we discuss how from Theorem~\ref{thm:get idempotent for main theorem} one can prove Theorems~\ref{main theorem introduction} and \ref{main introduction 2, Y(M) controllable} following ideas from Solecki's paper. The main novelty introduced here is the proof that $\YY$-controllable monoids and Ramsey monoids are aperiodic.

Let us recall some relevant notions for this section. 
Given a monoid $M$, the set \emph{$\YY(M)$}$\subseteq \powerset (\XX(M))$ consists of the non-empty subsets of $\XX(M)$ which are linearly ordered by inclusion. Define \emph{$x\leq_{\YY(M)}y$}, for $x, y\in \YY(M)$, if and only if $x\subseteq y$ and all elements of $y\setminus x$ are larger 
with respect to $\subseteq$ than all elements of $x$. 

There is a natural left $M$-action on $\YY(M)$ defined as $x\mapsto mx=\{maM: aM\in x\}$. 
Given two left actions of $M$ on $U$ and $U'$, a map $f:U\to U'$ is said \emph{$M$-equivariant} if it preserves the action of $M$, i.e. $f(ma)=mf(a)$.

For ease of notation, we isolate the following class of monoids.

\begin{definition}
	A monoid $M$ is called \emph{good} if for every left action of $M$ by continuous endomorphisms on a compact right topological semigroup $U$ there exists a function $g\colon {\mathbb Y}(M)\to E(U)$ such that 
	\begin{enumerate}
		\item[(i)] $g$ is $M$-equivariant;
		\item[(ii)] $g$ is order reversing with respect to $\leq_{{\mathbb Y}(M)}$ and $\leq_U$;
		\item[(iii)] $g$ maps maximal elements of ${\mathbb Y}(M)$ to $I(U)$.
	\end{enumerate}
\end{definition}

The notion of good monoids was first used by Solecki in \cite{MR3904213}. We borrow here three results that are contained or essentially proved therein.

The following useful lemma has the same function as two other lemmas by Lupini \cite[Lemma 2.2]{MR3620731} and Barrett \cite[Theorem 5.8]{barrett2020ramsey}, i.e. to get stronger conclusions from results like Theorem~\ref{theorem monoid actions introduction}.

\begin{lemma}[{\cite[Lemma 2.5]{MR3904213}}] \label{lemma Solecki} Let $M$ be a finite monoid. Assume that for every left action of $M$ by continuous endomorphisms on a compact right topological semigroup $U$ there is a $M$-equivariant $f$ from $\YY(M)$ to $U$ such that $f$ maps maximal elements of $\YY(M)$ to $I(U)$. Then, $M$ is good.
\end{lemma}

We isolate the following lemma from the proof of \cite[Theorem~2.4]{MR3904213} since it gives a sufficient condition for a monoid to be good. 

\begin{lemma}\label{step for theorem 4.4}
	Let $M$ be a finite  monoid and assume that for every action by continuous endomorphisms of $M$ on a compact right topological semigroup $U$ there exists a minimal idempotent $ {u}\in E(U)\cap I(U)$ such that $a( {u})=b( {u})$ for all couples $a,b\in M$ such that $a\Rbin b$. Then, $M$ is good.
\end{lemma}

\begin{proof}
	Let $\pi :\YY(M) \rightarrow \XX(M)$ be the function that maps a set $x\subseteq \YY(M)$ to the maximal element in $x$ with respect to $\subseteq$. Let $u\in E(U)\cap I(U)$ be given by hypothesis. The function $f:\XX(M) \rightarrow E(U)$ that maps $aM$ to $a(u)$ is well-defined, and maps $1M$ to $u\in E(U)\cap I(U)$. Also, notice that if $y$ is a maximal element of $\YY(M)$, then $1M\in y$ and so $\pi\circ f (y)=u\in E(U)\cap I(U)$. Since both $f $ and $\pi$ are $M$-equivariant the map $f \circ \pi: \YY(M)\rightarrow E(U)$ satisfies the assumptions of Lemma~\ref{lemma Solecki}, from which we get that $M$ is good.
\end{proof}

Solecki in \cite[Corollary 4.3]{MR3904213} states that every finite almost $\R$-trivial monoid is $\YY$-controllable, but in the proof he shows something stronger. In fact, the hypothesis that $M$ is almost $\R$-trivial is used only to apply \cite[Theorem 2.4]{MR3904213}, which states that every finite almost $\R$-trivial monoid is good. The remaining part of the proof never uses this hypothesis again, and relies instead on the fact that $M$ is good.  
In other words, from the proof of \cite[Corollary 4.3]{MR3904213} one can derive also the following result.

\begin{theorem}\label{thm: condition implies YY(M) controllable}
	Let $M$ be a finite monoid. If $M$ is good, then it is $\YY$-controllable.
\end{theorem}

However, the reader can find a short model-theoretic proof of this result in Section~\ref{section:proof of main theorem with model theory}.

Finally, the following is a restatement of part of the proof of \cite[Corollary~4.5 (i)]{MR3904213}. 

\begin{fact}\label{fact if X(M) is linear and conclusions of weakly Ramsey then M is Ramsey}
	If $M$ is $\YY$-controllable and $\XX(M) $ is linear, then $M$ is Ramsey.
\end{fact}

\begin{proof}
	Notice that $\XX(M)$ is linear if and only if $\XX(M)\in\YY(M)$. We want to use the definition of $\YY$-controllable with $\textbf{y}=\XX(M)$ and $F=\{\textbf{y}\}$.
	It is enough to notice that for every $a\in M$ we have
	\[ a\XX(M)=\{amM: mM\in \XX(M)\}=\{xM:  xM\subset aM\}.\] 
	Hence, if $aM\subseteq bM$, then $a\textbf{y}\vee b\textbf{y}=b\textbf{y}=b\textbf{y}\vee a\textbf{y}$, and so $a_1\textbf{y}\vee \dots \vee a_n\textbf{y}=\textbf{y}\in F$ for every $a_1,..,a_n\in M$ with at least one $i$ such that $a_i=1$.
\end{proof}

\begin{theorem}\label{main 2}
		Let $M$ be a finite monoid. \begin{enumerate}
		    \item If $M$ is aperiodic and has a linear $\XX_\R(M)$, then it is $\YY$-controllable. \label{main2-1} 
		    \item If $M$ is $\YY$-controllable, then it is aperiodic. \label{main2-2} 
		\end{enumerate}
\end{theorem}

\begin{proof} 
First, let $M$ be a finite aperiodic monoid with linear $\XX_\R(M)$.	By Theorem~\ref{thm:get idempotent for main theorem} and Lemma~\ref{step for theorem 4.4}, we get that $M$ is good. Hence, Theorem~\ref{thm: condition implies YY(M) controllable} implies that $M$ is $\YY$-controllable, and statement~\ref{main2-1} holds.
 
In order to prove \ref{main2-2}, let $(M^{<\omega},\conc)$ be the free semigroup over $M$, with coordinate-wise action. Notice that $(M^{<\omega},\conc)$ can be seen as $\langle (X_n)_{n<\omega}\rangle $ setting all $X_n=M$, with $1$ as distinguished point, and a word $w$ has a distinguished point if and only if $1\in\ran{w}$ (in which case we call $w$ a variable word).

Suppose $M$ is not aperiodic, and let $a\in M$ be such that $a^{n+1}\neq a^{n}$ for every $n\in\omega$. 
Let $A=\{a^{n}: n\in\omega\}$, and let $C=\{m\in M : a^nm\in A\mbox{ for some } n\in\omega\}$. Then, we have $ac\neq c$ for every $c\in C$, and $ac\in C$ if and only if $c\in C$.
	
	Let $\textbf{y}=\{a^{n}M: n\in\omega\}$, where we set $a^0=1$, and let $F=\{\textbf{y}\}$. 
	Then, $\textbf{y}$ is a maximal element of $\YY(M)$, and $\textbf{y}\vee \textbf{y}=\textbf{y}=a\textbf{y}\vee \textbf{y}$. 
	
	Let $C\cup\{\bot\}$ be the set of colors. Given a word $w\in M^{<\omega}$, let $m$ be the first letter of $w$ in $C$, if any. 
	If there is such $m$, color $w$ by $m$. Otherwise, color $w$ by $\bot$.   
	Consider any sequence of variable words $\bar{y}\in (M^{<\omega})^\omega$, and consider the words $ y_0 \conc y_1$ and $  {a}(y_0) \conc y_1$ with colors $c_1$ and $c_2$ respectively. Then, $c_1\in C$, since $y_0$ is a variable word and $1\in \ran(y_0)$. 
	Hence, by definition of $C$ we have $c_2=ac_1$. Therefore, $c_2=ac_1\neq c_1$, contradicting the fact that $M$ is $\YY$-controllable.
\end{proof}

\begin{theorem}\label{main 1}
	\label{thm:main theorem for Ramsey}
	Let $M$ be a finite monoid. The following are equivalent:
	\begin{enumerate}
		\item $M$ is Ramsey.\label{main-1}
		\item $M$ is aperiodic and $\XX(M)$ is linear.\label{main-2}
	\end{enumerate}
\end{theorem}
\begin{proof}
Proof of point \ref{main2-2} of Theorem~\ref{main 2} also shows that if $M$ is Ramsey then it is aperiodic. If $M$ is Ramsey, then $\XX(M)$ is linear by \cite[Corollary 4.5 (ii)]{MR3904213}.
Theorem~\ref{main 2} and Fact~\ref{fact if X(M) is linear and conclusions of weakly Ramsey then M is Ramsey} prove that \ref{main-2} implies \ref{main-1}.
\end{proof}

\begin{corollary}\label{cor:equivalence_Ramsey_and_YM-controllable+XM-linear}
    Let $M$ be a finite monoid. Then, $M$ is Ramsey if and only if it is $\YY$-controllable and $\XX(M) $ is linear.
\end{corollary}

We conclude this chapter with a corollary concerning the definition of Ramsey monoid.
It is not clear to the authors whether the following result can be proved with methods similar to those developed in Section~\ref{section:Equivalent definitions of Ramsey monoid}, without passing through Theorem~\ref{thm:main theorem for Ramsey}.

Recall that a variable words is a word $w$ such that $1\in \ran(w)$.

	\begin{corollary}
	\label{prop:Ramsey_W_M_without_rapisly_increasing}
	Let $M$ be a finite monoid. The following are equivalent: \begin{enumerate}
		\item $M$ is Ramsey.\label{point-2 Ramsey_W_M_without_rapisly_increasing}
		\item For all finite coloring of  $M^{<\omega}$ there are two variable words $y_0$ and $y_1$ such that $M{y_0}\conc y_1$ is monochromatic.\label{cor:Ramsey_weakest_condition}
	\end{enumerate}
\end{corollary}
\begin{proof} The exact same proofs of \cite[Corollary 4.5 (ii)]{MR3904213} and point~\ref{main2-2} of Theorem~\ref{main 2} show that if condition~\ref{cor:Ramsey_weakest_condition} hold, then $M$ is aperiodic and $\XX(M)$ is linear. The rest follows from Theorem~\ref{thm:main theorem for Ramsey} and by definition of Ramsey monoid.
\end{proof}

 \section{A model-theoretic approach}\label{section:proof of main theorem with model theory} 

In this section, we give a short explicit proof of Theorem~\ref{thm: condition implies YY(M) controllable}. 
We shall use Proposition~\ref{prop: types are right topological semigroup}, which says that the space of types $S(G)$ over a semigroup $G$ is a compact right topological semigroup if we add some symbol to the  signature. This approach is discussed in \cite{colla2019ramsey} and is further developed here. 
 The one difference in exposition is that here we define a product between types, while in \cite{colla2019ramsey} we define a product between type-definable sets. We assume some basic knowledge of model theory. We refer the reader to Tent's and Ziegler's book \cite{MR2908005}.

In what follows, we consider a semigroup $G$ such that $M$ acts by endomorphisms  on $G$, and a monster model $\G$.

We say that a type $p({\bl x})$ is  \emph{finitely satisfied\/} in $G$
if every finite conjunction of formulas in $p({\bl x})$ has a solution 
in $G^{|{ x}|}$.
We write \emph{${ a}\cnonfork_G{ b}$} if $\tp (a/Gb)$ is finitely satisfied in $G$. 
In literature this relation is also denoted by $a\nonfork^U_G{ b}$.

We say that the tuple ${ \bar c}$ is a \emph{coheir sequence\/} of
$p({\bl x})$ over $G$ if $c_n\models p(x)$ and ${  c_n}\cnonfork_G{  \bar{c}_{\restriction n}}$ and  ${  c_{n+1}}\equiv_{G,\,{  \bar{c}_{\restriction n}}}{  c_n}$ for every $n<\omega$.
In particular, $\bar{c}$ is indiscernible over $G$, i.e. ${  \bar{c}_{\restriction I_0}}\equiv_G {  \bar{c}_{\restriction I_1}}$ 
   for every $I_0,I_1\subseteq \omega$ of equal finite cardinality.

The following is an easy well-known fact.
\begin{fact}
For every type $p(x)\in S(G)$ there is a coheir sequence of $p(x)$.
\end{fact}

In order to define a product between types, we need some stationarity. Here, we obtain it by adding sets to the signature.

Let $G$ be a semigroup on which a monoid $M$ acts by endomorphisms and let $L$ be its signature.
Let $L^+=L\cup\{A : A\subseteq G^n\text{, }n\in\omega \} \cup\{m:m\in M\}$ be the expansion of $L$ where the symbol $A$ is interpreted in $G$ as the set $A$, and $m$ is interpreted as the unary function $a\mapsto ma$.

From now on we consider $G$ with this augmented signature.

\label{products of types}
 For $a, b\in S(G)$ define \emph{$ a\cdot_G b$} as 
    $\tp (  a'\cdot   b' /G)$, for any 
      ${  a', b'}\in \G$ such that $ a'\models a$,  $b'\models b$ and 
      ${  a'}\cnonfork_G{  b'}$.

As usual, we will consider the compact topology on $S(G)$ generated by the basic open sets $\{t\in S(G): \phi(x)\in t\}$, for $\phi(x) \in L^+(G)$. 

\begin{proposition}\label{prop: types are right topological semigroup}
If $G$ is a model as above, then $(S(G), \cdot_G)$ is a compact right topological semigroup and the action of $M$ defined by $m\tp (a/G)=\tp (ma/G)$ is by continuous endomorphisms.
\end{proposition}
\begin{proof}
Proposition 4.4 and Remark 2.7 of \cite[]{colla2019ramsey} prove that if $G$ is considered with  signature $L^+$ then $a\cdot_Gb$ is well-defined for every $a,b\in S(G)$ and $(S(G),\cdot_G)$ is a semigroup. In \cite[ Proposition 6.3]{colla2019ramsey} it is proved that the action of $M$ on $(S(G), \cdot_G)$ defined by $m\tp (a/G)=\tp (ma/G)$ gives well-defined endomorphisms of $(S(G), \cdot_G)$.
It is straightforward to check that the maps $m\tp (a/G)$ are continuous.
 The one missing proof is that $x\mapsto x\cdot_Gr$ is continuous from $S(G)$ to $S(G)$, for any $r\in S(G)$. 
 
Let $b\in \G$ and let $q(x,y)$ be the type in $S(G)$ such that $a\models q(x,b)$ if and only if $a\cnonfork_G b$, i.e.
\[ q(x,y)=\{\phi({ x};{ y})
    : 
   \phi({\bl x}\,;{ y})\in L^+(G)
   \text{ and } G^{|{\bl x}|}=\phi(G^{|{\bl x}|}\,;{ b})\}.\]
Notice that if $b'\equiv_Gb$, then $a\models q(x,b')\Longleftrightarrow a\cnonfork_G b'$.

     Let $p(z)$ be a partial type over $G$ and let $r(y)=\tp (b/G)$. The type \[\exists y,z\  r(y) \wedge q(x,y)\wedge z=x\cdot y \wedge p(z)\] is satisfied by those $a\in \G$ such that $\tp (a/G)\cdot_Gr$ satisfies $p(z)$. Hence, the preimage of the closed set $p(S(G))$ is closed.
\end{proof}

We are ready to prove Theorem~\ref{thm: condition implies YY(M) controllable}. 
Let us introduce the following auxiliary definition to simplify the notation of the next proof.
 
\begin{definition}\label{definition F-controllable}
Let $F$ be a finite subset of the semigroup $\langle \YY(M)\rangle$, let $\textbf{y}$ be a maximal element in $\YY(M)$, and let $c$ be a finite coloring of a semigroup $S$ on which $M$ acts. We say that a sequence $\bar{s}\in S^{\leq \omega}$ is \emph{$(F,\textbf{y},c)$-controllable} if  for every $m, n\leq |\bar{s}|$ and for every $a_i, b_j\in M$  \underline{if} $ a_0\textbf{y}\vee \dots \vee a_n\textbf{y}$ belongs to $F$ and
$ a_0\textbf{y}\vee \dots \vee a_n\textbf{y}=b_0\textbf{y}\vee\dots\vee b_m \textbf{y}$,
\underline{then} $ a_0s_{i_0}\cdot {\dots} \cdot a_ns_{i_n} $
has the same color of $b_0s_{j_0}\cdot {\dots} \cdot b_ms_{j_m}$, for every $i_0< \dots <i_n$,  $j_0<\dots <j_m $.
\end{definition}

\begin{proof}[Proof of Theorem~\ref{thm: condition implies YY(M) controllable}]
Let $(X_n)_{n\in\omega}$ be a sequence of pointed $M$-sets on which $M$ acts uniformly, and let ${\perp}$ be not in $\bigcup_{n\in\omega}X_n$. Define $G=(\langle (X_n)_{n\in\omega} \rangle \cup \{{\perp}\}, {\conc} )$ to be the semigroup extending $(\langle (X_n)_{n\in\omega} \rangle,{\conc})$ defining $x\conc y={\perp}$  if $x\conc y$ is not defined in the partial semigroup $\langle (X_n)_{n\in\omega} \rangle$. In particular, $x\conc {\perp}={\perp}\conc x={\perp}$. We write $x\cpaw y$ if and only if $x\conc y\neq {\perp}$ or $x={\perp}$.

It is enough to prove that for every finite subset  $F$  of $\langle \YY(M) \rangle$, for every maximal element $\textbf{y}$ in $\YY(M)$, and for every ${c}$ finite coloring of $G$ there is a basic sequence  $\bar{s}\in (\langle (X_n)_{n\in\omega} \rangle)^\omega$ that is $(F,\textbf{y}, c)$-controllable and such that $s_n$ has a distinguished point for every $n\in\omega$.

 Let $L=\{{\conc}, {\cpaw} \}$ and consider $G$ with augmented signature $L^+$.
 Let \[{U }=\{{ p}\in S(G): G\cpaw{  p}\}\] where $G\cpaw p$ is a shorthand for $\{g\cpaw x: g\in G\}\subseteq p(x)$.

  Notice that $U$ is non-empty, since for every finite set $A\subseteq G$ there is a $b\in G$ such that $A\cpaw \tp (b/G)$. Also, ${\perp} \notin U$. 
  We claim that ${U }$ is a compact subsemigroup of $(S(G),{\cdot_G})$. Let $\G$ be a monster model in the language $L^+$. Let $a,b \in \G$  such that ${  \tp (a/G)}\in {U }$, $\tp (b/G)\in {U }$ and ${  a}\cnonfork_G{  b}$. Then, we must have that ${  a}\cpaw{   b}$, since $G\cpaw { \tp (b/G)}$ and ${  a}\cnonfork_G{  b}$. Now, let $g\in G$ such that $g\neq {\perp}$. Then, $g\conc a\cnonfork_G b $ and hence $g\conc a\cpaw b$. Since $g\neq {\perp}$ and $g\cpaw a$ we also have $g\conc a \neq {\perp}$. 
  Hence, \[g\conc (a\conc b)=(g\conc a)\conc b\neq {\perp}.\] Therefore, we have that $g\cpaw (a\conc b)$ from which we get   
   ${  \tp (a/G)}{\cdot_G}{  \tp (b/G)}\in{ U }$. 
 Also, $ U $ is type-definable over $G$ and hence compact. Finally, notice that $a\cpaw mb$ for every $m\in M$ and for every $a,b\in G$ such that $a\cpaw b$. Hence, $U $ is closed under the action of $M$. By Proposition~\ref{prop: types are right topological semigroup},  it is a compact right topological semigroup such that $M$ acts on $U$ by continuous endomorphisms.

Let $u=g(\textbf{y})\in E(U)\cap I(U)$, where $g:\YY(M)\rightarrow E(U)$ is the function given by definition of good monoid. Let $DP$ be the set of elements of $\langle (X_n)_{n\in\omega} \rangle$ that have at least one distinguished point, and let $J=\{p\in S(G): DP\in p\}$.  Since $J\cap U$ is a non-empty both-sided ideal of $U$, and $u\in I(U)$, we have that $u$ is in $J$.
Let $(u_n)_{n\in\omega}$ be a coheir sequence of $u$. We write ${  \cev{u}_{\restriction i}}$ for the tuple 
${  u_{i-1}},\dots,{  u_{0}}$.  Notice that since the map $g$ is order-reversing and $M$-equivariant, for every $a_0,\dots, a_n, b_0,\dots,b_m\in M$ if $a_0\textbf{y}\vee \dots \vee a_n\textbf{y}=b_0\textbf{y}\vee\dots\vee b_m \textbf{y}$ then also $a_0 u\cdot_G {\dots} \cdot_G a_n u=b_0 u\cdot_G {\dots}\cdot_G b_m u$. Hence, it is straightforward to check that $  \cev{u}_{\restriction i}$ is $(F,\textbf{y}, c)$-controllable for every $i\in\omega$.  Notice that $  \cev{u}_{\restriction i}$ is a basic sequence since products stay in $U$ and ${\perp} \notin U$. Finally, $u_n$ are elements of $DP$ since $u\in J$. 
Now, we use the sequence ${ \cev u}$ to define $\bar{s} \in G^\omega$ with same properties as $\cev{u}$.

Let $k\in\omega$ be such that for every element $f\in F$ there are $k'< \kappa$ and $a_0,\dots,a_{k'}\in M$ such that $f=a_0\textbf{y}\vee\dots\vee a_{k'} \textbf{y}$. Notice that this implies that for every $f,f'\in \langle \YY(M)\rangle$ such that $f\vee f'\in F$ there are $c_0, \dots c_j\in M$ with $j< k$ such that $f'=c_0\textbf{y} \vee \dots \vee c_j\textbf{y}$. This follows from the property that the set of predecessors of any element of  $\YY(M)$ is linearly ordered by $\leq_{{\YY}(M)}$. 

Assume as induction hypothesis that the tuple obtained by concatenation $ s_{\restriction i}\conc{  \cev{u}}_{\restriction k}$ is $(F,\textbf{y}, c)$-controllable and $s_{\restriction i}$ is a  basic sequence of elements of $DP$.
Our goal is to find $s_i\in G$ such that the same properties hold for $s_{\restriction i+1}$.

   From the induction hypothesis it follows that
   $ s_{\restriction i}\conc{ \cev{u}_{\restriction l}}$ is $(F,\textbf{y}, c)$-controllable for any $l\in\omega$.
   In fact, let $w=b_0{s_{i_0}}\conc\dots\conc b_m {s_{i_m}}\conc b_{m+1}{u_{i_{m+1}}}\conc \dots\conc b_{n}{u_{i_{n}}} $ be such that $b_0{\textbf{y}}\vee\dots\vee b_m {\textbf{y}}\vee b_{m+1}{\textbf{y}}\vee \dots\vee b_{n}{\textbf{y}}\in F$.
   Let $j< k$ and $c_0,\dots,c_{j}\in M$ be such that $b_{m+1}\textbf{y}\vee \dots \vee b_{n} \textbf{y} = c_0\textbf{y}\vee\dots\vee c_{j} \textbf{y}$. 
   Since $\cev{u}$ is a coheir sequence, we have that $\cev{u}_{\restriction I_0}\equiv_G\cev{u}_{\restriction j+1}$ for any $I_0\subset l$ of size $j+1$. Hence,
   the type over $G$ of $w$ is equal to the type over $G$ of $b_0{s_{i_0}}\conc\dots\conc b_m {s_{i_m}}\conc c_{0}{u_{j}}\conc \dots \conc c_{j}{u_{0}}$. Therefore, we may use the induction hypothesis to conclude that $s_{\restriction i}\conc{ \cev{u}_{\restriction l}}$ is $(F,\textbf{y}, c)$-controllable.  Also, $s_{\restriction i}\conc{ \cev{u}_{\restriction l}}$ is a basic sequence by induction hypothesis and idempotence of $u$.

   Let $\phi(s_{\restriction i}, {  u_{i+1}}, {  u_{\restriction i+1}})$ 
   say that 
   $ s_{\restriction i}\conc{ \cev{u}_{\restriction i+2}}$ is $(F,\textbf{y}, c)$-controllable and that $s_{\restriction i}\conc { \cev{u}_{\restriction i+2}}$ is a   basic sequence of elements of $DP$.
   As ${ \bar u}$ is a coheir sequence we can find  $s_i\in G$ such that
   $\phi(s_{\restriction i+1}, {  u_{\restriction i+1}})$.
   Hence, 
   $s_i$
   has the desired properties.
\end{proof}

\section{Equivalent definitions of Ramsey monoid}\label{section:Equivalent definitions of Ramsey monoid}

In this section, we briefly prove the equivalence between different notions of being Ramsey. 

First, in Proposition~\ref{prop:transfer lemma for semigroups} we show that in this context results on located words are not stronger than results on words, but rather equivalent, since one can derive results about located words from results about words. While the converse is well-known, this implication apparently has been overlooked. 
For example, Bergelson-Blass-Hindman theorem on located words \cite{MR1262304} can be derived from Carlson's theorem on variable words \cite{MR926120}, since Carlson's Theorem implies condition \ref{def:Ramsey_W_M} in Proposition~\ref{prop:transfer lemma for semigroups}. Conditions \ref{def:Ramsey_W_M} and \ref{def:Ramsey_FIN_M} of Proposition~\ref{prop:transfer lemma for semigroups}  also show that Carlson's Theorem and Gowers' Theorem are indeed equivalent to the statement that a certain monoid is Ramsey. 

Secondly, in Corollary~\ref{cor:theorem maximum a_i} and Theorem~\ref{Milliken-Taylor} we state some equivalent definitions of Ramsey monoid which may be useful for applications.

First, let us recall some basic definitions. 
$(S,\cdot)$ is a
{partial semigroup} if $\cdot$ is a partial binary function $\cdot:S^2\rightarrow S$ such that $(s_1 \cdot s_2)\cdot s_3=s_1 \cdot (s_2\cdot s_3)$ whenever $(s_1 \cdot s_2)\cdot s_3$ and $s_1 \cdot (s_2\cdot s_3)$ are both defined. 
An endomorphism on a partial semigroup $S$ is a function $m : S \rightarrow S$, denoted by $s\mapsto ms$, such that for all
$s_1, s_2 \in S$ for which $s_1\cdot s_2$ is defined, then $ms_1 \cdot ms_2 $ is defined and $m(s_1\cdot s_2) = (ms_1) \cdot (ms_2 )$.

Given a partial semigroup $S$ and two sequences $\bar{s}$ and $\bar{t}$ in $S^\omega$, we say that $\bar{s}$ is \emph{extracted} from $\bar{t}$, or \emph{$\bar{s}\leq_M \bar{t}$}, if there is an increasing sequence $(i_n)_{n\in\omega}$ of natural numbers such that $s_n\in {\langle t_{i_n},\dots, t_{(i_{n+1})-1}\rangle}_M$. 
As for pointed $M$-sets, we say that $\bar{t}$ is \emph{basic} if $m_0 t_{i_0}\cdot{\dots}\cdot {m_n t_{i_n}}$ is defined for every $i_0<{\dots}<i_n$ and $m_0,{\dots },m_n\in M$. 

An infinite sequence $\bar{t}\in M^\omega$ is \emph{rapidly increasing} if $\displaystyle |t_n | > \Sigma_{i=0}^{n-1}|t_i |$ for all $n\in\omega$. 
Let $\FIN_M$ be the partial semigroup of located words, i.e. $\FIN_M=\langle (X_n)_{n\in\omega}\rangle$ where $X_n=\{n\}\times M$ with the usual action. As for words, a variable located word is a located word $w$ such that $(n,1_M)\in\ran{w}$ for some $n$.

\begin{proposition} \label{prop:transfer lemma for semigroups}
	The following are equivalent for a monoid $M$:
	\begin{enumerate}    
		\item $M$ is Ramsey. \label{def:Ramsey}
		\item For every partial semigroup $S$ on which $M$ acts by endomorphisms, for every basic sequence $\bar{t}\in S^\omega$, for every finite coloring of $S$ there is a sequence $\bar{s}\leq_M \bar{t}$ such that $\langle \bar{s} \rangle_M$ is monochromatic.\label{def:Ramsey_with_quadruple}
		\item There is a rapidly increasing sequence of variable words $\bar{x}\in {(M^{<\omega})}^\omega$ such that for all finite colorings of $M^{<\omega}$ there is a sequence $\bar{s}\leq_M \bar{x}$ with $\langle \bar{s} \rangle_M$ monochromatic.\label{def:Ramsey_W_M} 
		 \item For all finite colorings of $\FIN_M$ there is a sequence of variable located words with monochromatic $M$-span.\label{def:Ramsey_FIN_M}  
		\end{enumerate}
\end{proposition}

\begin{proof}
    It is easy to check that points~\ref{def:Ramsey} and \ref{def:Ramsey_with_quadruple} are equivalent. Indeed, $\langle (X_n)_{n\in\omega}\rangle$ is a partial semigroup for every $(X_n)_{n\in\omega}$, hence \ref{def:Ramsey_with_quadruple} implies \ref{def:Ramsey}. Conversely, given a partial semigroup $S$ and a basic sequence $(t_n)_{n\in\omega}\in S^{\omega}$, we may obtain results about $S$ from $\langle (X_n)_{n\in\omega}\rangle$ choosing the sets $X_n=Mt_n\subset S$ with distinguished point $t_n$.

	It is also straightforward to check that \ref{def:Ramsey_with_quadruple} implies  \ref{def:Ramsey_W_M} and \ref{def:Ramsey_FIN_M}. Hence, it remains to prove that \ref{def:Ramsey_W_M} implies \ref{def:Ramsey_with_quadruple} and \ref{def:Ramsey_FIN_M} implies \ref{def:Ramsey_with_quadruple}.
	
	So let us show that \ref{def:Ramsey_W_M} implies \ref{def:Ramsey_with_quadruple}.
	Let $\bar{x}$ be the sequence given by \ref{def:Ramsey_W_M}, let $S$ be a partial semigroup and let $\bar{t}$ be a basic sequence in $S^\omega$. We write $\conc $ to denote the operation of $M^{<\omega}$ and $\bullet$ to denote the operation of $S$.
	
	By definition of rapidly increasing sequence and since each ${x_n}$ is a variable word, every element of $\langle \bar{x}\rangle_M$ can be written as ${{m}_{1}}x_{i_1} \conc {\dots} \conc {{m}_{n}}x_{i_n}$ in a unique way. Hence, there is a function $f:\langle  \bar{x}\rangle_M \rightarrow \langle \bar{t}\rangle_{M}$ defined as the surjective homomorphism of partial semigroups such that for every $n\in\omega$, ${m}_1\dots {m}_{n}\in M$, and $i_1< \dots <i_n$,
	\[f({{m}_{1}}x_{i_1} \conc {\dots} \conc {{m}_{n}}x_{i_n})={{m}_1}t_{i_1}\bullet{\dots}\bullet {{m}_n}t_{i_n}.\]

	Let $\{B_i : i<n\}$ be a coloring of $S$ into finitely many pieces. Define $A_i=f^{-1}[{B_i}]$, then $\{A_i : i<n\}\cup \{M^{<\omega}\setminus \langle \bar{x}\rangle_M\}$ is a finite coloring of $M^{<\omega}$. By \ref{def:Ramsey_W_M} we may find $\bar{y}\leq_M \bar{x}$ such that $\langle \bar{y}\rangle_ M$ is monochromatic. Notice that $\langle \bar{y}\rangle_ M\subset \langle \bar{x}\rangle_ M$, so there is $k<n$ such that
	$\langle \bar{y}\rangle_ M\subset A_k$.
	
	Set $f(\bar{y})=(f(y_i))_{i\in\omega}$.  Then, $f(\bar{y})\leq_{ M }f(\bar{x})=\bar{t}$. It is enough to prove that $\langle f(\bar{y})\rangle_ { M }= f[\langle \bar{y}\rangle_ { M}]$ and then we are done, since $f[\langle \bar{y}\rangle_ { M}]\subset f[A_k]= B_k$. 
	
	Notice that ${m} f( y_j)=f({m} y_j)$ for all $m\in M$. In fact, if $y_j={{m}_1}x_{i_1}\conc{\dots}\conc {{m}_n}x_{i_n}$, then \[{m} f(y_j)={m}({m_1}t_{i_1})\bullet{\dots}\bullet  {m}({m_k}t_{i_n})=f({m}y_j).\] 
	
	Let $g\in \langle f(\bar{y})\rangle_ { M }$, say \[g={{m}_1}f(y_{i_1})\bullet{\dots}\bullet {{m}_n}f(y_{i_n})=f({{m}_1}y_{i_1})\bullet{\dots}\bullet f({{m}_n}y_{i_n}).\] Then,  $g=f({{m}_1}y_{i_1}\conc{\dots}\conc {{m}_n}y_{i_n})$.
	
	The proof of \ref{def:Ramsey_FIN_M} implies \ref{def:Ramsey_with_quadruple} proceeds in a similar manner, starting from the sequence $\bar{x}=((n,1_M))_{n\in\omega}$ from which every sequence of variable located words can be extracted.
\end{proof}

Notice that in the proof of Proposition~\ref{prop:transfer lemma for semigroups}, with the same notation and assumptions, we showed that the color of $g={{m}_1}f(y_{i_1})\bullet{\dots}\bullet {{m}_n}f(y_{i_n})$ is controlled by the color of ${{m}_1}y_{i_1}\cdot{\dots}\cdot {{m}_n}y_{i_n}$. 
Proposition~\ref{prop:transfer lemma for semigroups} can be extended to $\YY$-controllable monoids, providing three equivalent definitions for this notion as well.

In the definition of $M$-span, we ask that at least one element of the basic sequence is moved by 1. 
Here, we show how to relax this condition.

\begin{corollary}\label{cor:theorem maximum a_i}
Let $M$ be a finite Ramsey monoid. Then, for any partial semigroup $S$, for any finite coloring of $S$ and for every sequence $\bar{t}\in S^\omega$ there is $\bar{s}\leq_M \bar{t}$ such that for every $a\in M$ the set  \[\big\{m_{0}\,{s_{i_0}}\cdots m_{n}\,{s_{i_n}}
:  
i_0<\dots<i_n, m_i\in aM, m_i\Rbin a \mbox{ for at least one } i
\big\}\]  is monochromatic.
\end{corollary}

\begin{proof}
By Corollary~\ref{cor:equivalence_Ramsey_and_YM-controllable+XM-linear}, $M$ is $\YY$-controllable. Then, the thesis follows from the definition of $\YY$-controllable monoid applied to the maximal element $\textbf{y}=\XX(M)$ and to $F=\{a\textbf{y}: a\in M\}$, and the arguments of Proposition~\ref{prop:transfer lemma for semigroups}.
\end{proof}

In previous corollary, the action of $M$ can be controlled with $|\XX(M)|$-many colors. This is optimal, as in general it is not possible to get less than $|\XX(M)|$-many monochromatic sets. For example, choose $\XX(M)$ as set of colors, and color each word $w\in M^{<\omega}$ by the minimum $aM$ such that $\ran(w)\subset aM$: then, if $\bar{t}$ is a sequence of variable words, for any $\bar{s}\leq_M\bar{t}$ each set defined above has a different color.

When instead $M$ is $\YY$-controllable but $\XX(M)$ is not linear, it is not difficult to see that for any  $k\in\omega$ there are $\textbf{y}$ and $F\subset \YY(M)$ and colorings of, say, $M^{<\omega}$ such that for every sequence of variable words $\bar{s}$ there are more than $k$-many $f\in F$ such that the sets 
\[\langle \bar{s}\rangle_f=\{ a_0s_{i_0}\cdot {\dots} \cdot a_ns_{i_m} : a_i\in M, i_0<\dots<i_{m}, a_0\textbf{y}\vee \dots \vee a_n\textbf{y}=f\}
\] 
have different colors.

The next theorem is a generalization of both Theorem~\ref{main 1} and Milliken-Taylor theorem \cite{MR373906}, \cite{MR424571}. It is a combination of Ramsey's theorem 
and Theorem~\ref{main 1}, in the same way as Milliken-Taylor theorem is a combination of Ramsey's theorem and Hindman's theorem. For a sequence $\bar{s}\in S^\omega$ let $\bar{s}^{(n)}$ be the collection of $n$-subsets of $\{a\in S: a=s_i\mbox{ for some }i\in\omega\}$. Notice that for $n=1$ the following is the content of Theorem~\ref{main 1}.

\begin{theorem}\label{Milliken-Taylor}
	Let $M$ be a finite Ramsey monoid. 
	Then, for any $n\geq 1$, for all sequences of pointed $M$-sets $(X_n)_{n\in\omega}$ on which $M$ acts uniformly, for any finite coloring of $n$-subsets of $\langle (X_n)_{n\in\omega} \rangle$ there is a basic sequence $\bar{s}\in (\langle (X_n)_{n\in\omega} \rangle)^{\omega}$ such that $s_n$ has a distinguished point for every $n\in\omega$ and such that $\displaystyle \bigcup_{\bar{r}\leq_M\bar{s}}{\bar{r}}^{(n)}$ is monochromatic.
\end{theorem}	
\begin{proof}
 The proof goes as in Theorem~\ref{thm: condition implies YY(M) controllable}, in section \ref{section:proof of main theorem with model theory}. Let  $G$ and $u=g(\textbf{y})$ be defined as in  Theorem~\ref{thm: condition implies YY(M) controllable}, with  $\textbf{y}=\XX(M)$, and let $(u_n)_{n\in\omega}$ be a coheir sequence of $u$. It is straightforward to check that all elements of the span of $  \cev{u}_{\restriction i}$ satisfy the type $u$ for every $i\in\omega$. Also, notice that with signature $L^+$, for every $a,a',b\in \G^{<\omega}$
    	if $a\equiv_G a'$,  $a'\cnonfork_G b$, and $a\cnonfork_G b$, then $a\equiv_{Mb}a'$.
    	Then, for any $\cev{h}\leq_M \cev{u}$ we have $\cev{h}\equiv_G \cev{u}$, by the remark above and the definition of coheir sequence.
    	All the $n$-subsets of an indiscernible sequence have the same color, for any $n\in\omega$. The rest of the proof is the same as in Theorem~\ref{thm: condition implies YY(M) controllable}.
\end{proof}
 
 The same arguments of  Proposition~\ref{prop:transfer lemma for semigroups} allow to extend this result to any partial semigroup.

It can be easily seen that if a monoid satisfies the conclusions of Corollary~\ref{cor:theorem maximum a_i} or the conclusions of Theorem~\ref{Milliken-Taylor}, then it is Ramsey.  Conversely, Corollary~\ref{cor:theorem maximum a_i} and Theorem~\ref{Milliken-Taylor} hold for all finite Ramsey monoids. Hence, their conclusions hold for a finite monoid if and only if it is Ramsey.

\section{Final remarks and open problems}\label{section: open questions}

 We conclude with some open questions and remarks concerning the work done so far.

Our main theorems suggest a possible connection between Ramsey theory and automata theory, passing through Schützenberger's Theorem. Any result in that direction would be of the highest interest.
	
Limiting ourselves to Ramsey theory, there are still several challenging open questions in the context of monoid actions on semigroups. 

One of the most immediate questions is what can be proven for infinite monoids. In a fore-coming paper, the authors prove that there are not infinite Ramsey monoids, and thus a monoid is Ramsey if and only if it is finite, aperiodic and has linear $\XX(M)$.

	 Theorem~\ref{main 2} provides a sufficient condition for a monoid to be $\YY$-controllable. 
	This condition is not necessary, as there are $\YY$-controllable monoids for which $\XX_\R(M)$ is not linear.
	
	\begin{proposition}\label{example YY controllable monoid not linear}
     Let $M$ be a finite aperiodic monoid such that  for every distinct $a,b\in M$ with $a\Rbin b$, we have $a^2=a$ and  $ax=bx$ for every $x\in M\setminus\{1\}$. Then, $M$ is $\YY$-controllable. 
    \end{proposition}
    
    \begin{proof}
     To show that $M$ is $\YY$-controllable is enough to work with $M^{<\omega}$, by the arguments of Proposition~\ref{prop:transfer lemma for semigroups}.
     
    Consider the monoid $\tilde{M}=(M,*)$ where $x*y=y$ for all $x,y\neq 1$. It acts coordinate-wise  on $M^{<\omega}$, considered as $\tilde{M}^{<\omega}$. 
    
     Let $G$ be $M^{<\omega}$ with the signature $L^+$ used in the proof of Theorem \ref{thm: condition implies YY(M) controllable}, plus an unary function $\tilde{a}$ for any $a\in M$, which is interpreted in $G$ as the action of $\tilde{M}$. Since $\tilde{M}$ is Ramsey and since every element in $\tilde{M}$ different from $1$ is in the same $\R$-class, one can find an idempotent $u$ in the space of types $S(G)$ such that $\tilde{a}u=\tilde{b}u$ for every $a,b \neq 1$. Let $v$ be an element of the monster model satisfying $u$. Then, if $a\Rbin b$, we have \[av=a\tilde{a}v\equiv_G a\tilde{b}v=bv,\]
where we use the fact that for every $x\in M^{<\omega}$, and hence for every $x$ in the monster model, we have $a\tilde{a}x=ax$ and $a\tilde{b}x=bx$, by hypothesis. Hence, we can conclude that $M$ is $\YY$-controllable, by the arguments of Theorem \ref{thm: condition implies YY(M) controllable}. 	  
    \end{proof}

	An example of a monoid satisfying the hypothesis of Proposition \ref{example YY controllable monoid not linear} for which $\XX_\R(M)$ is not linear is given by the following Cayley table.

    \begin{table}[h]
        \centering
        \begin{tabular}{c|c c c c }
             {1} & $a$ & $b$ & $c$ & $d$  \\
             \hline
             $a$ & $a$ & $b$ & $a$ & $b$ \\
             $b$ & $a$ & $b$ & $a$ & $b$  \\
             $c$ & $c$ & $d$ & $c$ & $d$  \\
             $d$ & $c$ & $d$ & $c$ & $d$ 
             \end{tabular}
        \caption{Example of (aperiodic) $\YY$-controllable monoid $M$ such that  $\XX_\R(M)$ is not linear.}
        \label{tab:example_monoid_2}
    \end{table}

		On the other hand, it seems possible that the necessary condition of Theorem \ref{main 2} is also sufficient, and that a finite monoid is aperiodic if and only if it is $Y$-controllable.
		If true, this would suggest an even stronger connection between Ramsey theory and Schützenberger's Theorem.
	
	\begin{OP}
	 Find an algebraic characterization of $\YY$-controllable monoids.
	\end{OP}

		  If $M$ is a Ramsey monoid, then \textit{for every action of $M$ on every partial semigroup} you have a monochromatic set as described in the definition.
		    Lupini's in \cite{MR3620731} gave examples of non-Ramsey monoids where the same statement holds for \textit{certain actions on certain partial semigroups} (actually, he proved a stronger statement that can be seen as the analogue of Corollary~\ref{cor:theorem maximum a_i}).

			Define {$I_k$} to be the set of functions $f$ from $k$ to $k$ such that $f(0)= 0$ and such that if 
			$f({i})= j$ then either $f({i+1})= j$ or $f({i+1})= {j+1}$. Then, {$I_k$} is a monoid with composition of functions as operation, and {$k$} is an $I_k$-set with distinguished point $k-1$, where the action is defined by $fi=f(i)$. 
			This action induces a coordinate-wise action on {$\FIN_k=\langle (\{n\}\times k)_{n\in\omega}\rangle$} (i.e. the set of all partial functions with finite domain from $\NN$ to $k$). 
			Lupini in \cite{MR3620731} showed that {for every $k\in\omega$ and} for every finite coloring of {$\FIN_k$} there is an {infinite} sequence of words in {$\FIN_k$}  each containing $k-1$ such that its {$I_k$}-span is monochromatic. Notice that this result implies that every $\R$-trivial monoid is Ramsey.
			In fact, let $N$ be a $\R$-trivial monoid with linear $\XX(N)$. Without loss of generality, we may assume that $N=\{0,\dots, k-1\}$ and that $0N\subseteq \dots \subseteq (k-1) N$ is an increasing enumeration of $\XX(N)$. 
			Then, the coordinate-wise action of $N$ on $\FIN_k$ coincides with the action of a submonoid of $I_k$, by Proposition~\ref{prop:Green_relations_for_XM_linear}, and Lupini's theorem implies point~\ref{def:Ramsey_FIN_M} of Proposition~\ref{prop:transfer lemma for semigroups}.
			
			All monoids $I_k$ are $\R$-trivial, but $\XX(I_k)$ is linear if and only if $k\leq 3$ (see \cite[Section 4 .4]{MR3904213}). In particular, if $k>3$ these monoids are not Ramsey, and Lupini's result does not follow from the theory of Ramsey monoids. It would be interesting to see if a similar statement holds for other (non-Ramsey) monoids.

						\begin{OP}
			Classify the couples $(M,k)$ such that $k\in \omega$ is a pointed $M$-set and for every finite coloring of $\FIN_k$ there is a basic sequence $\bar{s}$ in $\FIN_k$  such that $s_n$ has a distinguished point for every $n\in\omega$ and
			such that the $M$-span of $\bar{s}$ is monochromatic.
			\end{OP}

			In the same direction, the following seems a challenging problem. 
			
			\begin{OP}
				Characterize the class of triples $(S,M, \bar{t})$, where $S$ is a partial semigroup, $M$ is a monoid acting on $S$ by endomorphisms and $\bar{t}$ is a basic sequence in $S$, for which for every finite coloring of $S$ there is a sequence $\bar{s}\leq_M\bar{t}$ in $S$ such that its $M$-span is monochromatic.
		\end{OP} 
	 
 One can check that every finite Ramsey monoid generates examples of Ramsey spaces. However, an even nicer property might be true: there are topological Ramsey spaces that induce a collection of projected spaces such that every metrically Baire set has the Ramsey property. A sufficient condition for the latter has been found by Dobrinen and Mijares in \cite{MR3366475}. 	An example of a space of this form is Carlson-Simpson space, see \cite{MR753869} and \cite[section 5.6]{MR2603812}). See also \cite[section 4]{MR3366475} for generalizations of the latter.

	\begin{OP}
	Which topological Ramsey spaces given by finite Ramsey monoids meet the sufficient conditions given in \cite{MR3366475}?
	\end{OP}

 Hales-Jewett theorem \cite{MR143712} is a corollary of Corollary~\ref{cor:theorem maximum a_i} for the special case of monoids $M$ such that $ab=b$ for every $a,b \in M\setminus\{1\}$. In Ramsey theory, two of the strongest known results are a polynomial generalization \cite{MR1715320} and a density generalization \cite{MR1191743} of Hales-Jewett theorem for these monoids.

 \begin{OP}
 Do polynomial or density results hold for other finite Ramsey monoids? 
 \end{OP}

Ojeda-Aristizabal in \cite{MR3638339}  obtained upper bounds for the finite version of Gowers' $\FIN_k$ theorem, giving a constructive proof.
It would be interesting to know if these upper bounds hold for other Ramsey monoids.

The work of Gowers on $\FIN_k$ and the related space $\FIN_{\pm k}$ was the key to his solution of an old problem in Banach spaces \cite{MR1164759}.
Also, the aforementioned example of Bartošova and Kwiatkowska found applications in metric spaces. Finally, a discussion about the connection between Ramsey spaces and Banach spaces can be found in Todorcevic's monograph. In this paper, we found new Ramsey monoids, and consequently new Ramsey spaces. Hence, it might be possible to find applications of these new results to metric spaces.

	 Recently various papers have found different common generalizations of Carlson's and Gowers' theorems, see  \cite{barrett2020ramsey}, \cite{kawach2020parametrized}, \cite{MR3896106}. Of particular interest is the context of adequate layered semigroups,  introduced by Farah, Hindman, and McLeod \cite{MR1899627} and recently studied by Lupini  \cite{MR3896106} and Barrett  \cite{barrett2020ramsey}. 
	 	 Barrett's paper \cite{barrett2020ramsey} describes a framework which seems well suited for a connection between Ramsey monoids and layered semigroups. 
	 His work and ours are independent from each other and were written concurrently, so we did not investigate this research line. 
	 Nevertheless, in Example~\ref{example Barret} and in the following paragraph  we show a possible connection.

 \begin{example}\label{example Barret} 
	 Let $M$ be a monoid with linear $\XX(M)$, and let $a_0M \subseteq \dots \subseteq a_nM $ be an increasing enumeration of $\XX(M)$. 
	 Define $\ell:\FIN_M\to n+1$ by \[\ell(w)=\min\{i: \ran(w)\subset a_iM \}.\]
	 Then, $(\FIN_M,\ell)$ is an adequate partial layered semigroup as defined in \cite[Definition~3.7]{barrett2020ramsey}.
	 Furthermore, the canonical action $\mathcal{F}_{\cw}$ of $M$ on $\FIN_M$ is made of regressive maps, by Proposition~\ref{prop:Green_relations_for_XM_linear}. 
	 \end{example}
	 
	 This example shows that every monoid with linear $\XX(M)$ generates an adequate partial layered semigroup, $\FIN_M$, and a family of regressive functions $\mathcal{F}_{\cw}$. 
	 On the other hand, every family of regressive functions $\mathcal{F}$ on an adequate partial layered semigroup generates a monoid $M_\mathcal{F}$ with composition, acting on $S$ by endomorphisms.


      \begin{center}
    \Large{\textbf{Acknowledgements}}
\end{center}
We would like to thank our supervisors Luca Motto Ros and Domenico Zambella for the help given, and Turin logic group, which provides a good environment for our formation. The second author would like to thank the Jerusalem logic group, which was very welcoming last year. Also, we would like to thank Natasha Dobrinen for interesting conversations around topological Ramsey spaces. We thank the anonymous referee for valuable suggestions on the structure of this presentation.

    \bibliographystyle{abbrv}
    \bibliography{Bibliography}{}

Universit\`a degli Studi di Torino,
Dipartimento di Matematica ``G. Peano'',
Via Carlo Alberto 10, 10123 Torino, Italy.\\
\textit{Email address:}  claudio.agostini@unito.it

Universit\`a degli Studi di Torino,
Dipartimento di Matematica ``G. Peano'',
Via Carlo Alberto 10, 10123 Torino, Italy.\\
\textit{Email address:} eugenio.colla@unito.it

\end{document}